\documentclass{amsart}

\usepackage{amsfonts,amssymb,amsmath,amsthm}
\usepackage{url}
\usepackage{enumerate}

\usepackage{graphics}
\usepackage{graphicx}

\newtheorem{theorem}{Theorem}
\newtheorem{lemma}{Lemma}

\newtheorem{definition}{Definition}

\newtheorem{corollary}{Corollary}
\newtheorem{remark}{Remark}
\newtheorem{example}{Example}

\def\ret{{\rm{ Ret }}}
\def\ult{{\rm{ Ult }}}
\def\no{{\rm{ NO }}} 
\def\lo{{\rm{ LO }}}    
\def\age{{\rm{ Age }}}
\def\aut{{\rm{ Aut }}}
\def\homeo{{\rm{ Homeo }}}

\def\exp{{\rm{ Exp }}}
\def\card{{\rm{ card }}}

\title{Universal minimal flows of groups of automorphisms of uncountable structures}
\author{Dana Barto\v{s}ov\'a}
\address{Department of Mathematics\\
University of Toronto\\
Bahen Center\\
40 St. George St.\\
Toronto\\
Ontario\\
Canada\\
M5S 2E4}
\email{dana.bartosova@utoronto.ca}


\begin{document}
\maketitle

\textbf{Abstract.} It is a well-known fact, that the greatest ambit for a topological group $G$ is the Samuel compactification of $G$ with respect to the right uniformity on $G.$  We apply the original destription by Samuel from 1948 to give a simple computation of the universal minimal flow for groups of automorphisms of uncountable structures using Fra\"iss\'e theory and  Ramsey theory. This work generalizes some of the known results about countable structures.

\section{Introduction}
Universal minimal flows play an important role in topological dynamics. In \cite{KPT}, the authors explicitely compute universal minimal flows of groups of automorphisms of countable structures with certain properties using Fra\"iss\'e theory and Ramsey theory. In the same spirit, we compute universal minimal flows for dense subgroups of automorphism groups of uncountable structures. However, for uncountable structures we do not have in hand a generic ordering, so we have to take a slightly different route. This work was initially inspired by a talk of Y.Gutman \cite{YG} on the universal minimal flow of the group of homeomorphisms of $\omega^*=\beta\omega\setminus\omega,$ where $\beta\omega$ is the \v{C}ech-Stone compactification of discrete $\omega,$ and a suggestion of Todor\v{c}evi\'c to reprove his result using the ideas of \cite{KPT}, a project that has been also the subject of his joint work with P. Ursino several years ago. We observe that the same space serves as the universal minimal flow for all dense subgroups, hence for all normal subgroups of the groups of trivial homeomorphisms of $\omega^*$ described by van Douwen in \cite{vD}.

The structure of the paper is as follows: In the first section, we introduce the basic notions from topological dynamics and describe them for the case of groups that admit a basis of neighbourhoods of the neutral element of open subgroups. In the second section, we talk about groups of automorphisms and their relationship with the groups in the first section. In the third section, we introduce crucial ingredients - Fra\"iss\'e classes, Ramsey theory and linear orderings. In the third section, we prove the main theorem and apply it to describe some universal minimal flows. In the fourth section, we use the main theorem to characterize extremely amenable groups, generalizing Theorem $4.3$ from \cite{KPT} to uncountable structures.

\section{Topological dynamics}

The central notion of topological dynamics is a continous action $\pi:G\times X\to X$ of a topological group $G$ on a compact Hausdorff space $X.$ We call $X$ a \emph{$G$-flow} and omit $\pi$ if it is understood and write $gx$ instead of $\pi(g,x).$ A \emph{homomorphism} of $G$-flows $X$ and $Y$ is a continuous map $\phi:X\to Y$ respecting the actions of $G$ on $X$ and $Y$, i.e. $\phi(gx)=g\phi(x)$ for every $g\in G,$ $x\in X$ and $y\in Y.$ We say that $Y$ is a \emph{factor} of $X,$ if there is a homomorphism from $X$ onto $Y.$ Every $G$-flow has a minimal subflow, a minimal closed subspace of $X$ invariant under the action of $G.$ Among all minimal $G$-flows, there is a maximal one - the \emph{universal minimal flow} $M(G).$ It means that every other minimal $G$-flow is a factor of $M(G).$ In the study of universal minimal flows, a construction of the greatest ambit turns out to be useful. An \emph{ambit} is a $G$-flow $X$ with a distinguished point $x_0\in X$ whose orbit $Gx_0=\{gx_0:g\in G\}$ is dense in $X.$ Likewise for minimal flows, there is a maximal ambit - the \emph{greatest ambit} $(S(G),e).$ It means that every other ambit $(X,x_0)$ is a factor of $S(G)$ via a quotient mapping sending $e$ to $x_0.$ As we say below, the greatest ambit is a compactification of $G$ with a structure of a right-topological semigroup and the universal minimal flow is a minimal left ideal of $S(G).$ For introduction to topological dynamics see \cite{dV}.

\subsection*{Greatest ambit} 
If $G$ is a discrete group, then its greatest ambit is the ultrafilter dynamical system $\ult(G).$ It is the space of all ultrafilters on $G$ with the topology generated by clopen sets $A^*=\{u\in\ult(G):A\in u\}.$ The action of $G$ on $\ult(G)$ extends the multiplication in $G:$ $gu=\{gA:A\in u\}$ for $g\in G$ and $u\in\ult(G).$ Moreover, we can extend the multiplication to all of $\ult(G),$ turning $\ult(G)$ into a semigroup: for $u,v\in \ult(G),$ we set $uv=\{A\subset G: \{g\in G: g^{-1}[A]\in v\}\in u\}.$ In other words, $uv=u-{\rm lim} \{gv:g\in G\}.$ Fixing an ultrafilter $u,$ the right multiplication $\cdot u: \ult(G)\to \ult(G),$ $v\mapsto uv$ is continuous, hence $\ult(G)$ is a \emph{right-topological semigroup}.

If $G$ is a topological group, then the greatest ambit of $G$ is a factor of $\ult(G)$ by definition. It has first been described by Samuel in 1948 (see \cite{S}) in the setting of uniform spaces and seems to be overlooked by topological dynamists. We shortly remind his construction for topological groups with the right uniformity:
Let $G$ be a topological group with the neutral element $e$ and $\mathcal{N}$ a basis of neighbourhoods of $e$ giving the topology of $G.$ The \emph{right uniformity} on $G$ is generated by covers $\{Va:a\in G\}$ for $V\in\mathcal{N}.$ We define an equivalence relation $\sim$ on $\ult(G)$ as follows:
For an ultrafilter $u\in\ult(G),$ we define a filter $u^*$ to be generated by $\{VA:A\in u, V\in\mathcal{N}\}.$ Then for $u,v\in\ult(G)$ we set $u\sim v$ if and only if $u^*=v^*.$ The quotient space $S(G)=\ult(G)/\sim$ with the quotient topology is called the Samuel compactification of $G$ and it is the greatest ambit of $G.$ The multiplication on $\ult(G)$ also factors to $S(G),$ making $S(G)$ a right-topological semigroup with the multiplication extending the multiplication on $G.$

\textbf{From now on},  $G$ will always be a topological group that possesses a basis $\mathcal{N}$ of open neighboorhoods of the neutral element $e$ consisting of open (hence clopen) subgroups of $G$. It is easy to see  that then $L=\{VA:A\subset G, V\in \mathcal{N}\}$ is a Boolean algebra and we get the following description of the greatest ambit.

\begin{lemma}
The greatest ambit $S(G)$ is equal to the Stone space of the Boolean algebra $L$ as above with the action defined by $gx=\{gA:A\in x\}$ for $g\in G$ and $x\in \ult(L).$ 
\end{lemma}

\begin{proof}
Let $u, v \in\ult(G).$ Since $VV=V$ for every $V\in \mathcal{N},$  we have that $u\cap L$ is an ultrafilter on $L$ that generates $u^*$. It means that whenever $u\cap L=v\cap L,$ then $u\sim v.$ On the other hand, if $u\cap L\neq v\cap L,$ then there is $A\in L$ such that $A\in u\cap L$ and $G\setminus A\in v\cap L,$ which implies that $u^*\neq v^*.$ This gives us the sought for isomorphism between $S(G)$ and ultrafilters on $L.$
\end{proof}
 
The quotient multiplication on $\ult(L)$ can be described explicitely as an analogue of the multiplication on $\ult(G).$

\begin{lemma}
Let $u,v\in\ult(L),$ $A\subset G$ and $V\in\mathcal{N}.$ Then $VA\in uv$ if and only if $\{g\in G: g^{-1}[VA]\in v\}\in u.$ 
\end{lemma}

\begin{proof}
Let $u',v' \in\ult(G)$ be representatives of the equivalence classes of $u,v$ respectively. We need to show that the definition of $uv$ in the claim corresponds to $u' v' \cap L.$ By the definition of ultrafilter multiplication, $VA\in u' v' \cap L$ if and only if $U=\{g\in G: g^{-1}[VA]\in v'\}\in u'.$ First, we show that we can replace $u'$ with $u:$ Notice that $VU=U,$ since for every $s\in V$ and $g\in G,$ we have that $sg\in U$ if and only if $(sg)^{-1}[VA]=g^{-1}s^{-1}[VA]=g^{-1}[VA]\in v$ if and only if $g\in U.$ Therefore, $VA\in u' v' \cap L$ if and only if $U\in u' \cap L=u.$ Second, we show that $v'$ can be replaced by $v:$ For every $g\in G$ and $W\in\mathcal{N},$ $g^{-1}[VA]\in v'$ implies $Wg^{-1}[VA]\in v.$ Pick a $W\in\mathcal{N}$ such that $gWg^{-1}\subset V.$ Then $Wg^{-1}\subset g^{-1}V,$ so $Wg^{-1}VA\subset g^{-1}VVA=g^{-1}VA.$ But obviously, $g^{-1}VA\subset Wg^{-1}VA,$ hence $Wg^{-1}VA=g^{-1}VA$ which shows that $g^{-1}VA\in v'$ if and only if $g^{-1}VA\in v' \cap L=v,$ which concludes the proof.
\end{proof}

\subsection*{Universal minimal flow}

Let $G$ be a group with a basis of neighbourhoods $\mathcal{N}$ of the neutral element $e$ consisting of open subgroups and let $L$ be as above. The universal minimal flow $M(G)$ for $G$ being a subspace of the greatest ambit $S(G)$, is itself a Stone space. Hence we can consider its Boolean algebra of all clopen subsets $B(G).$ For $m\in M(G)$ and $\emptyset\neq U\in B(G),$ denote by $\ret(m,U)$ the set of elements of $G$ that return $m$ into $U,$ i.e. $\ret(m,U)=\{g\in G:gm\in U\}.$ Since $h\ret(m,U)=\ret(m,hU)$ and $M(G)$ is compact, there are finitely many $g_1,g_2,\ldots,g_n\in G$ such that $\bigcup_{i=1}^n g_i\ret(m,U)=G.$ Such sets are called \emph{syndetic}. More generally:

\begin{lemma}
The following are equivalent for a $G$-flow $X$
\item[(i)] $X$ is minimal,
\item[(ii)] for every non-empty open set $O\subset X,$ $\bigcup_{g\in G} gO=X,$
\item[(iii)] for every $x\in X$ and non-empty open set $O\subset X,$ the set $\ret(x, O)$ is syndetic.
\end{lemma}

\begin{proof}
$(i)\Rightarrow (ii)$
Let $X$ be a minimal $G$-flow, $x\in X$ and $O$ a non-empty open subset of $X.$ If $X\setminus \bigcup_{g\in G}gO\neq\emptyset,$ then it is a non-trivial closed subflow witnessing non-minimality of $X$. 

$(ii)\Rightarrow (iii)$
Since $X$ is compact, the cover $\{gO:g\in G\}$ has a finite subcover $\{g_1O,g_2O,\ldots,g_nO\}.$ Then $\bigcup_{i=1}^n g_i\ret(x,O)=\bigcup_{i=1}^n\ret(x,g_iO)=G,$ so $\ret(x,O)$ is syndetic.

$(iii)\Rightarrow (i)$
If $X$ is not minimal, then there is an	$x\in X$ such that $O=X\setminus \overline{Gx}$ is a non-empty open set. But then $\ret(x,O)=\emptyset,$ hence not syndetic.
\end{proof}

A Boolean algebra of subsets of $G$ is called a \emph{syndetic algebra}, if it is invariant under left translations by elements from $G$ and all of its non-empty elements are syndetic sets.

Now, we are ready to imitate a proof for discrete semigroups from \cite{BF} to characterize $B(G)$.

\begin{theorem}
The universal minimal flow $M(G)$ is the Stone space of a maximal syndetic subalgebra of $L.$ All maximal syndetic subalgebras of $L$ are isomorphic.
\end{theorem}

\begin{proof}
Let $B(G)$ denote the algebra of clopen subsets of $M(G)$ and let $m\in M(G).$ Since $M(G)$ is a minimal left ideal of $S(G),$ $S(G)m=M(G),$ so the right translation $R_m$ by $m$ maps $S(G)$ onto $M(G)$. $S(G)$ being a right-topological semigroup, $R_m$ is continuous and being morover onto,  it induces a dual injective homomorphism $\rho_m:B(G)\to L.$ By the claim above, we know that for every $VA\in L,$ $\rho_m((VA)^*\cap M(G))= \{u\in S(G):VA\in um\}=\{g\in G: g^{-1}[VA]	\in m\}=\ret(m, (VA)^*).$ So $B(G)$ is isomorphic to $\mathcal{A}=\{\ret(m, (VA)^*): A\subset G, V\in\mathcal{N}\},$ hence a subalgebra of $L$ consisting of syndetic sets only. Now, we show that $\mathcal{A}$ is invariant under left tranlations by $G:$ Let $h\in G,$ then $h\ret(m, (VA)^*)=h\{g\in G:g^{-1}[VA]\in m\}=\{x\in G:x^{-1}[gVA]\}=\ret(m,(gVA)^*).$ It remains to show that $\mathcal{A}$ is a maximal syndetic algebra. Let $\mathcal{B}\supset \mathcal{A}$ be a syndetic algebra. Then $\ult(\mathcal{B})$ with multiplication $gu=\{gA:A\in u\}$ for $g\in G$ and $u\in \ult(\mathcal{B})$ is a minimal flow. The identity embedding of $\mathcal{A}$ into $\mathcal{B}$ induces a $G$-homomorphism from $\ult(\mathcal{B})$ onto $\ult(\mathcal{A})\cong M(G).$ By universality, $\ult(\mathcal{B})\cong \ult(\mathcal{A}),$ which in turn implies $\mathcal{A}\cong\mathcal{B}.$

For the other part of the theorem, let us assume that $\mathcal{A}$ is a maximal syndetic subalgebra of $L.$ To achieve the conclusion, it is enough to find an $m\in M(G)$ such that $\mathcal{A}=\{\ret(m, O):O\in B(G)\}.$ $\ult(\mathcal{A})$ with $G$-multiplication as above is a minimal $G$-flow. Hence there is a $G$-homomorphism $\phi:M(G)\to \ult(\mathcal{A}).$ Consider $p\in \ult(\mathcal{A})$ given by $p=\{A\in \mathcal{A}:e\in A\}$ and its preimage $m$ under $\phi.$ Denote by $\hat{A}$ the clopen set $\{u\in \ult(\mathcal{A}):A\in u\}.$ Then $\ret(p,\hat{A})=\ret(m,\phi^{-1}[A]),$ since $\phi(gm)=g\phi(m)=gp$ for all $g\in G.$ But also $\ret(p,\hat{A})=\{g\in G: gp\in\hat{A}\}=\{g\in G: g^{-1}[A]\in p\}=\{g\in G: e\in g^{-1}[A]\}=\{g\in G: ge\in A\}=A.$ So we have that $\mathcal{A}=\{\ret(p,\hat{A}):A\in \mathcal{A}\}\subset \{\ret(m, O):O\in B(G)\}.$ By maximality of $\mathcal{A},$ we get that $\mathcal{A}=\{\ret(m, O):O\in B(G)\}.$

\end{proof}

\begin{corollary}
$M(G)$ is a totally disconnected space.
\end{corollary}

\section{Automorphism groups}
Let $\kappa$ be a cardinal number endowed with the discrete topology and denote by $S_{\kappa}$ the group of all bijections on $\kappa.$ In what follows, we consider $S_{\kappa}$ as a topological group with the topology of pointwise convergence. The topology is given by a basis of neighbourhoods of the neutral element consisiting of open subgroups $S_A=\{g\in S_{\kappa}:g(a)=a, a\in A\}$ where $A$ is a finite subset of $\kappa.$ We can observe that a subset  $H$ of $S_{\kappa}$ is closed if and only if it contains every $g\in S_{\kappa}$ such that for any $A\subset\kappa$ finite there exists an $h\in H$ with $g|A=h|A.$

Let $G\leq S_{\kappa}$ be a  subgroup and let $A$ be a finite subset of $\kappa.$ We denote the point-wise stabilizer of $A$ as :
$$
G_A=\{g\in G: ga=a, a\in A\}
$$
and the set-wise stabilizer as:
$$
G_{(A)}=\{g\in G: gA=A\}.
$$

Let $L$ be a first order language and $\mathcal{A}$ an $L$-structure with the universe of cardinality $\kappa.$ Then the group of automorphisms of $\mathcal{A},$ $\aut(\mathcal{A})$ is a closed subgroup of $S_{\kappa}.$ It is clear that the topology on $\aut(\mathcal{A})$ is given by  $\aut(\mathcal{A})_A$ for $A$ a finitely-generated substructure of $\mathcal{A}.$ If every partial isomorphism between two finite subsets of $\mathcal{A}$ can be extended to an automorphism of the whole structure $\mathcal{A},$ then we say that $\mathcal{A}$ is \emph{$\omega$-homogeneous}. In what follows, when we say ``a structure'', we mean a structure for some first order language.

We finally describe the correspondence between groups possessing a neighbourhood basis of the neutral element of open subgroups, dense subgroups of groups of automorphisms of structures and subgroups of $S_{\kappa}.$

\begin{theorem}
Let $G$ be an infinite topological group and let $\kappa$ be a cardinal number. Then the following are equivalent:
\begin{itemize}
\item[(a)] $G$ is a subgroup of $S_{\kappa},$
\item[(b)] $G$ has a basis $\mathcal{N}$ of neighbourhoods of the neutral element of cardinality $\lambda\leq \kappa$ consisting of open subgroups such that the family of all left translates of elements from $\mathcal{N}$ also has cardinality $\lambda,$
\item[(c)] $G$ is a dense subgroup of a group of automorphisms of an $\omega$-homogeneous relational structure on a set of cardinality $\kappa,$
\item[(d)] $G$ is a dense subgroup of a group of automorphisms of a structure on a set of cardinality $\kappa.$
\end{itemize}
\end{theorem}

\begin{proof}
The equivalence of $(a),(c)$ and $(d)$ follows from the Theorem $4.1.1$ in \cite{H}. $(b)$ trivially follows from $(a),$ so we only need to establish that also $(b)$ implies $(a).$ We proceed as in Theorem $1.5.1$ in \cite{BK}. Let $\{U_i:i\in\lambda\}$ be an enumeration of a basis for the topology on $G$ given by all left translates of subgroups from $\mathcal{N}.$ For every $g\in G,$ define $\phi(g)=\pi_g\in S_{\kappa}$ as follows
$$
\pi_g(i)=j \Leftrightarrow gU_i=U_j {\rm \ for\ } i\in\lambda {\rm \ and \ } \pi_g(i)=i {\rm \ otherwise.}
$$
Obviously, $g\mapsto \pi_g$ is an injective homomorphism of $G$ into $S_{\kappa}.$ To show that it is continuous, let $A\subset\lambda$ be finite and let  $S_{\kappa,A}=\{\pi\in S_{\kappa}:\pi(i)=i, i\in A\}$ be a basic open subgroup of $S_{\kappa}.$	Then $\phi^{-1}(S_{\kappa,A})=\{g\in G:gU_i=U_i, i\in A\}.$ Since for every $i\in A,$ $U_i=h_iV_i$ for some $h_i\in G$ and some subgroup $V_i\in\mathcal{N},$ $\phi^{-1}(S_{\kappa,A})=\bigcap_{i\in A}\{g\in G: gh_iV_i=h_iV_i\}=\bigcap_{i\in A}\{g\in G:h_i^{-1}gh\in V_i\}=\bigcap_{i\in A}\{g\in G: g\in h_iV_ih_i^{-1}\},$ which is an intersection of finitely many open set, hence open. Similarly, let $H\in\mathcal{N}$ be an open subgroup of $G,$ so $H=U_i$ for some $i\in \lambda.$ Then $\phi(H)=\{\pi_h:h\in H\}=\{\pi\in S_{\kappa}:\pi(i)=i\}\cap \phi(G),$ hence $\phi$ is a homeomorphism onto its image.
\end{proof}

\section{Fra\"iss\'e classes, Ramsey theory and linear orderings}

Now, we give necessary definitions and facts about  Fra\"iss\'e classes, the Ramsey property for finite structures and linear orderings. For a comprehensive treatise of these ingredients see \cite{KPT}.

\subsection*{Fra\"iss\'e classes}\label{F}
A class of finitely-generated structures $\mathcal{F}$ of a given language is called a \emph{Fra\"iss\'e class}, if it satisfies the following conditions:
\begin{itemize}
\item[(HD)] Hereditary property :  if $A$ is a finitely generated substructure of $B$ and $B\in\mathcal{F},$ then also $A\in \mathcal{F}.$
\item[(JEP)] Joint embedding property : if $A,B\in \mathcal{F}$ then there exists a $C\in\mathcal{F}$ in which both $A$ and $B$ embed.
\item[(AP)] Amalgamation property: if $A,B,C\in\mathcal{F}$ and $i:A\to B$ and $j:A\to C$ are embeddings, then there exist $D\in\mathcal{F}$ and emebddings $k:B\to D$ and $l:C\to D$ such that $k\circ i=l\circ j.$
\end{itemize}

Let $\mathcal{A}$ be an $\omega$-homogeneous structure. Then it is easily verified that $\age(\mathcal{A}),$ the class of all finitely-generated substructures of $\mathcal{A},$ is a Fra\"iss\'e class. In case of countable structures, there is a one-to-one correspondence between $\omega$-homogeneous structures and Fra\"iss\'e classes (\cite{F}).

\begin{example} \label{E} The following are Fra\"iss\'e classes: 
\begin{itemize}
\item[(a)] finite sets
\item[(b)] linearly ordered finite graphs
\item[(c)] finite Boolean algebras
\item[(d)] finite vectors spaces over a finite field
\item[(e)] finite linear orderings
\end{itemize}
\end{example}

In what follows, we will only be interested in Fra\"iss\'e classes consisting of finite structures.

\subsection*{Ramsey theory}

 A class $\mathcal{K}$ of finite structures satisfies the \emph{Ramsey property} if for every $A\leq B\in\mathcal{K}$ and $k\geq 2$ a natural number there exists $C\in\mathcal{K}$ such that 
 $$
 C\to (B)^A_k,
 $$
 i.e. for every colouring of copies of $A$ in $C$ by $k$ colours, there is a copy $B'$ of $B$ in $C,$ such that all copies of $A$ in $B'$ have the same colour.

\begin{example}
All examples of Fra\"iss\'e classes in Example \ref{E}  satisfy the Ramsey property: $(a)$ is the classical Ramsey theorem, $(b)$ was proved by Abramson and Harrington \cite{AH} and Ne\v{s}et\v{r}il and R\"odl  \cite{NR}, $(c)$ is equivalent to the so-called dual Ramsey theorem by Graham and Rothschild \cite{GR}, $(d)$ was proved by Graham, Leeb, Rothschild in \cite{GLR} and $(e)$ is equivalent to $(a).$
\end{example}

\subsection*{Linear orderings}
Let $\lo(\kappa)$ denote the space of all linear orderings on $\kappa$ considered as a subspace of $2^{\kappa\times\kappa}$ with the product topology. The topology on $\lo(\kappa)$ is generated by 
$$(A,<)^*=\{<'\in \lo(\kappa):<'|A=<\},$$
for $A\subset\kappa$ finite and $<$ a linear ordering on $A.$

Let $\mathcal{A}$ be a structure of size $\kappa$ and let again $\age(\mathcal{A})$ denote the family of finitely-generated substructures of $\mathcal{A}.$ Let $L$ be the language of $\mathcal{A}$ and let $L'=L\cup \{<\}$ be an expansion of $L$ by a binary relational symbol $<$ not in $L.$ Let $\mathcal{K}$ be a family of structures for $L'$ in which $<$ is a linear ordering. Suppose that $\age(\mathcal{A})$ is a reduct of $\mathcal{K},$ i.e. $\age(\mathcal{A})=\{K|L:K\in\mathcal{K}\}.$ Let $\prec\in\lo(\kappa)$ be such that $(A,\prec|A)\in\mathcal{K}$ for every $A\in\age(\mathcal{A}).$ Then we call $\prec$ a normal ordering of $\mathcal{A}$ induced by $\mathcal{K}.$ The space of all normal orderings of $\mathcal{A}$ induced by $\mathcal{K}$ is denoted by $\no_{\mathcal{K}}(\mathcal{A})$ and it is a closed subspace of $\lo(\kappa).$

\begin{example}[\cite{KPT}] \label{natural}
\noindent
\begin{itemize}
\item[(BA)] Let $\mathcal{F}$ be the class of all finite Boolean algebras. We call a linear ordering on a finite Boolean algebra \emph{natural}, if it is induced antilexicographically by a linear ordering on its atoms. Let $\mathcal{K}$ denote the class of all naturally ordered finite Boolean algebras. If $\mathcal{B}$ is a homogeneous Boolean algebra, then $\no(\mathcal{B})\neq\emptyset,$ since both $\age(\mathcal{B})=\mathcal{F}$ and $\mathcal{K}$ are Fra\"iss\'e classes. Moreover, they satisfy the Ramsey property.
\item[(VS)] Let $\mathcal{F}$ be the class of all finite vector spaces over a given finite field $\mathcal{G}$. As in the previous example, we call a linear ordering on a finite vector space of $\mathcal{G}$ natural, if it is an antilexicographical ordering induced by a linear ordering on a basis and a fixed linear ordering of the field $\mathcal{G}$ (see \cite{T}). Let $\mathcal{K}$ be the class of all naturally ordered finite vector spaces. If $\mathcal{V}$ is an infinite vector space over $\mathcal{G},$ then $\no(\mathcal{V})\neq\emptyset,$ since $\mathcal{V}$ is $\omega$-homogeneous and both $\age(\mathcal{V})=\mathcal{F}$ and $\mathcal{K}$ are Fra\"iss\'e classes, the letter shown by Thomas in \cite{T}. Moreover, they satisfy the Ramsey property (see \cite{KPT},p.$144$).
\end{itemize}
\end{example}

Every subgroup $G$ of $S_{\kappa}$ has a natural action on $\lo(\kappa)$ given by $a (g<) b$ if and only if $g^{-1}a < g^{-1} b$ and in the same way if $\mathcal{A}$ is a structure, $\mathcal{K}$ as above and $H$ a subgroup of $\aut(\mathcal{A}),$ then $H$ has a natural action on $\no_{\mathcal{K}}(\mathcal{A}).$

\section{Computations of universal minimal flows}
Let $\mathcal{A}$ be an $\omega$-homogeneous structure with finitely-generated substructures finite and let $G$ be a dense subgroup of the group of automorphisms of $\mathcal{A}.$ Let $A$ be a finite substructure of $\mathcal{A}.$ Then we can identify the right coset space $G_{(A)}/G$ with copies of $A$ in $\mathcal{A}$ via $G_{(A)}g \longleftrightarrow g^{-1}[A]\in \binom{A}{A}.$ Similarly, we can identify the right cosets of $G_A$ in the  right cosets of $G_{(A)}$ with automorphisms of respective copies of $A.$

Now assume that $\mathcal{K}$ is a Fra\"\i ss\'e order class whose reduct is $\age(\mathcal{A}).$ Following \cite{KPT}; Definition 5.5., we say that $\mathcal{K}$ is \emph{order forgetful} whenever 
\begin{equation}\label{eq}
(A,<),(B,\prec)\in \mathcal{K} \mbox{\ and\ } A\cong B \mbox{\ imply \ } (A,<)\cong (B,\prec).
\end{equation}

It is shown in Proposition 5.6 of \cite{KPT} that such an order forgetful expansion $\mathcal{K}$ of $\age(\mathcal{A})$ has the Ramsey property if and only if $\age(\mathcal{A})$ does.

In \cite{KPT}, the univeral minimal flows are computed for groups of automorphisms of countable $\omega$-homogeneous structures $\mathcal{A}$ whose $\age$ has an order expansion $\mathcal{K}$ satisfying the \emph{ordering property}, i.e. for every $A\in\age(\mathcal{A}),$ there is $B\in\age(\mathcal{A})$ such that whenever $\prec$ is a linear ordering on $A,$ $\prec'$ is a linear ordering on $B$ and $(A,\prec),(B,\prec')\in\mathcal{K},$ then $(A,\prec)\leq (B,\prec').$ Since the order forgetful expansions of $\age(\mathcal{A})$ trivially have the ordering property, the following result generalizes Theorem 7.5(ii) to uncountable structures in this special case.

\begin{theorem}\label{T}
Let $\mathcal{A}$ be an $\omega$-homogeneous structure. Suppose that finitely-generated substructures of $\mathcal{A}$ are finite and that they satisfy the Ramsey property. Suppose that $\mathcal{K}$ is a Fra\"\i ss\'e order class that is an order forgetful expansion of $\age(\mathcal{A}).$ 
Then $\no_{\mathcal{K}}(\mathcal{A})$ induced by $\mathcal{K}$ is the universal minimal flow for every dense subgroup of $\aut(\mathcal{A}).$
\end{theorem}

\begin{proof}
Let $G$ be a dense subgroup of $\aut(\mathcal{A}).$ To prove minimality of $\no_{\mathcal{K}}(\mathcal{A}),$ we need to verify that $\ret(<,(A,<')^*)=\{g\in G: g<\in(A,<')^*\}$ is syndetic for every $<\in\no_{\mathcal{K}}(\mathcal{A})$ and $(A,<')\in\mathcal{K}.$ Let $S^A$ be a set of representants for right cosets of $G_A$ in $G_{(A)}.$ As $\mathcal{K}$ is order forgetful, $\ret(<,(A,<')^*)$ intersects every right coset of $G_{(A)},$ so $S^A\ret(<,(A,<')^*)=G.$ Since $A$ is finite, also $S^A$ is finite and hence $\ret(<,(A,<')^*)$ is syndetic.

To prove universality, we need to show that given $<\in\no_{\mathcal{K}}(\mathcal{A}),$ the syndetic algebra $\mathcal{B}$ generated by $\{\ret(<,(A,<')^*):(A,<')\in\mathcal{K}\}$ is a maximal syndetic subalgebra of the algebra $L=\{G_AK:K\subset G, A\in\age(\mathcal{A})\}.$  

In order to do that, we show that if left translates of a set $H=G_AK\in L$ generate a syndetic subalgebra of $L,$ then $H$ intersects every right coset of $G_{(A)}:$
Suppose that $\{g:H\cap G_{(A)}g=\emptyset\}$ is nonempty. Then $H'=S^AH=G_{(A)}K$ is in the algebra generated by left translates of $H$ and $G\setminus H'\neq\emptyset.$ Consider the induced colouring $c: \binom{\mathcal{A}}{ A} \to \{H', G\setminus H'\}$ given by $c(A')=H'$ if and only if there is an $h\in H$ such that $h^{-1}[A]=A'.$ Suppose now that there are $g_1,g_2,\ldots,g_n\in G$ witnessing that both $H'$ and $G\setminus H'$ are syndetic. Let $C$ be the substructure of $\mathcal{A}$ generated by $\bigcup_{i=1}^n g_i[A].$ Since $\age(\mathcal{A})$ satisfies the Ramsey property, there is a copy $C'$ of $C$ in $\mathcal{A}$ which is monochromatic, say in the colour $H'.$ Since $G$ is dense in $\aut(\mathcal{A})$ and $\mathcal{A}$ is $\omega$-homogeneous, there is an $f\in G$ mapping $C'$ to $C.$ It means that $c(f^{-1}g_iA)=H'$ for all $i,$ which shows that there is no $h\in G\setminus H'$ and no $i$ such that $f=g_i\circ h,$ a contradiction.

Now let $H=G_AK\in L$ be syndetic. We have that $A_{\sigma}:=\sigma\ret(<,(A, <)^*)\in \mathcal{B}$ for every $\sigma\in S^A$ and $G=\bigcup_{\sigma\in S^A} A_{\sigma},$ therefore $\bigcup_{\sigma\in S^A}(H\cap A_{\sigma})=H.$ For every $\sigma\in S^A$ either $A_{\sigma}\cap H=\emptyset,$ or $A_{\sigma}\cap H=A_{\sigma},$ since $A_{\sigma}$ takes exactly one right coset of $G_A$ in every right coset of $G_{(A)}$ and then $H\in\mathcal{B},$ or $\mathcal{B}\cup \{H\}$ does not generate a syndetic algebra . It follows that $\mathcal{B}$ is maximal.   
 
\end{proof}

\subsection*{Group of all bijections}
Since both finite sets and finite linear orders are Fra\"iss\'e classes satisfying the  Ramsey property and finite sets are a reduct of finite linear orders satifying (\ref{eq}), we can generalize the result of Glasner and Weiss (\cite{GW2}) about $S_{\omega}$ to $S_{\kappa}$ for arbitrary infinite $\kappa.$

\begin{theorem}
The universal minimal flow of $S_{\kappa}$ is $\lo(\kappa).$
\end{theorem}

\subsection*{Homogeneous Boolean algebras}

$\omega$-homogeneous Boolean algebras are usually called just homogeneous. The Age of a homogeneous Boolean algebra is the class of all finite Boolean algebras. Recall (see \cite{KPT}; p.$145$) that a linear ordering $<$ on a finite Boolean algebra is natural if it is an antilexicographical order induced by an ordering of the atoms. Since both the class of finite Boolean algebras and the class of naturally ordered finite Boolean algebras are Fra\"iss\'e classes with the Ramsey property and they satisfy the assumptions of the theorem, we get the following result, which generalizes Theorem $8.2 (iii)$ in \cite{KPT} to uncountable homogeneous Boolean algebras.

\begin{theorem}
Let $\mathcal{B}$ be a homogeneous Boolean algebra and $\mathcal{K}$ the class of naturally ordered finite Boolean algebras. Then the universal minimal flow of $\aut(\mathcal{B})$ is the space $\no_{\mathcal{K}}(\mathcal{B})$ of all linear orderings on $\mathcal{B}$ that are natural when restricted to a finite subalgebra.
\end{theorem}

Homogeneous Boolean algbebras are in Stone duality with $h$-homogeneous zero-dimensional compact Hausdorff spaces, so this is just a dual version of a result by Glasner and Gutman. Recall that a topological space $X$ is called \emph{$h$-homogeneous}, if all non-empty clopen subsets of $X$ are homeomorphic.

\begin{theorem}[\cite{GG}]
Let $X$ be an $h$-homogeneous zero-dimensional compact Hausdorff topological space. Let $G=\homeo(X)$ equipped with the compact-open topology, then $M(G)=\Phi(X),$ the space of maximal chains on $X.$
\end{theorem}

The space of maximal chains was introduced by Uspenskij in \cite{U}: Let $X$ be a compact space and denote by $\exp X$ the space of closed subsets of $X$ equipped with the Vietoris topology. Then the space $\Phi(X)$ of all maximal chains of closed subsets of $X$ is a closed subspace of $\exp \exp X.$ The natural action of $\homeo(X)$ on $X$ induces an action on $\exp X$ and $\Phi(X),$ which is the action considered in the theorem above. There is of course an explicit isomorphism between these two universal minimal flows  (see Theorem 8.3 in \cite{KPT}). 

\medskip

Following a paper by van Douwen \cite{vD}, if $\kappa$ is a cardinal number, we denote  by $\mathcal{P}(\kappa)/[\kappa]^{<\kappa}$ the quotient algebra of the Boolean algebra of all subsets of $\kappa$ by the ideal of sets of cardinality less than $\kappa.$ It is easy to see that $\mathcal{P}(\kappa)/[\kappa]^{<\kappa}$ is homogeneous for every cardinal $\kappa.$ 

Now we introduce two subgroups of $\mathcal{P}(\kappa)/[\kappa]^{<\kappa}:$ 
Denote by $T_{\kappa}$ the set of all bijections between subsets $A,B\subset\kappa$ with $\card(\kappa\setminus A), \card(\kappa\setminus B)<\kappa.$ With the operation of composition, $T_{\kappa}$ is a monoid, but not a group. We can however assign to each $f\in T_{\kappa}$ an automorphism $f^*$ of $\mathcal{P}(\kappa)/[\kappa]^{<\kappa},$ $f^*([X])=[f[X]],$ mapping $T_{\kappa}$ onto a subgroup $T_{\kappa}^*=\{f^*:f\in T_{\kappa}\}$ of $\aut(\mathcal{P}(\kappa)/[\kappa]^{<\kappa}).$ Since the automorphisms in $T_{\kappa}^*$ are induced by a pointwise bijection between subsets of $\kappa,$ we call them \emph{trivial}. Inside of $T_{\kappa}^*$ we have a normal subroup of those automorphisms induced by a true permutation of $\kappa,$ let us denote it by  $S_{\kappa}^*.$ Shelah \cite{Sh} (see also \cite{SS}) proved that consistently every automorphism of $\mathcal{P}(\omega)/{\rm fin}$ is trivial. This has been extended to $\mathcal{P}(\kappa)/{\rm fin}$ for all cardinals $\kappa$ in \cite{V}. Of course, consistently the two groups are different. 

The next theorem shows that $T_{\omega}^*$ and $S_{\omega}^*$ do not coincide, hence $T_{\omega}^*$ (and thus consistently also $\aut(\mathcal{P}(\omega)/{\rm fin})$) is not simple.

\begin{theorem}[\cite{vD}]
There is a homomorphism $h^*$ from $T_{\omega}^*$ onto $\mathbb{Z}$ with kernel $S_{\omega}^*.$ (In particular, $T_{\omega}^*\neq S_{\omega}^*.$)
\end{theorem}

Van Douwen also identified all normal subgroups of $T_{\omega}^*.$ 

\begin{theorem}[\cite{vD}]
A subgroup $G$ of $T_{\omega}^*$ is normal if and only if $\card{(G)}=1$ or $G\in\{(h^*)^{-1}k\mathbb{Z}:k\in\mathbb{N}\}.$
\end{theorem}

It follows that all non-trivial normal subgroups of $T_{\omega}^*$ are dense in $\aut(\mathcal{P}(\omega)/{\rm fin}).$ So we can apply Theorem \ref{T} to obtain the following corollary.

\begin{corollary}
The universal minimal flow for all normal subgroups of $T_{\omega}^*$  and of $\aut(\mathcal{P}(\omega)/{\rm fin})$ is  $\no(\mathcal{P}(\omega)/{\rm fin}).$
\end{corollary}

Note that even though the algebraic structure of $S_{\kappa}^*$ is inherited from $S_{\kappa}$, their topologies are radically different, therefore so are their universal minimal flows $\no(\mathcal{P}(\kappa)/[\kappa]^{<\kappa})$ and $\lo(\kappa)$ respectively (even in cardinality).

\medskip

In the uncountable case, the situation is slightly different.
 
\begin{theorem}[\cite{vD}]
If $\kappa>\omega,$ then $T_{\kappa}^*=S_{\kappa}^*.$
\end{theorem}

Nevertheless, $T_{\kappa}^*$ is dense in $\aut(\mathcal{P}(\kappa)/[\kappa]^{<\kappa}),$ so Theorem \ref{T} applies.

\begin{corollary}
Let $\kappa$ be a cardinal number. Then the universal minimal flow of $\aut(\mathcal{P}(\kappa)/[\kappa]^{<\kappa})$ and $T_{\kappa}^*$ is $\no(\mathcal{P}(\kappa)/[\kappa]^{<\kappa}).$
\end{corollary}

\subsection*{Vector spaces over finite fields}
Every vector space over a finite field is $\omega$-homogeneous and its $\age$ is equal to all finite dimensional spaces that form a Fra\"iss\'e class satisfying the Ramsey property  and so does the class of all finite dimensional naturally-ordered spaces. Moreover, these classes satisfy the condition (\ref{eq}) of Theorem \ref{T}, as noted in \cite{KPT}, p.$144.$ Thus we can generalize the result in \cite{KPT} about countable-dimensional vector spaces to those with uncountable dimension.

\begin{theorem}
The universal minimal flow of $\aut(\mathcal{V}),$ where $\mathcal{V}$ is an infinite dimensional vector space over a finite field, is the space $\no(\mathcal{V}).$
\end{theorem}

\section{Extremely amenable groups}

Next theorem characterizes extremely amenable subgroups of $S_{\kappa}.$ The implications $(a)\Rightarrow(c)$ and $(c)\Rightarrow (b)$ have identical proofs as for $S_{\infty}$ given in \cite{KPT}.

\begin{definition}
A topological group is called extremely amenable, if its universal minimal flow is a singleton.
\end{definition}

\begin{theorem} \label{T2}
Let $G$ be a subgroup of $S_{\kappa}.$ The following are equivalent:
\begin{itemize}
\item[(a)] $G$ is extremely amenable,
\item[(b)] $(i)$ for every finite $A\subset\kappa,$ $G_A=G_{(A)}$ and $(ii)$ for every colouring $c:G/G_A\to\{1,2,\ldots, k\}$ and  for every finite $B\supset A,$ there is $g\in G$ and $i\in\{1,2,\ldots,n\}$ such that $c(hG_A)=i$ whenever $h[A]\subset g[B].$
\item[(c)] $(i')$ $G$ preserves an ordering and $(ii)$ as above.
\end{itemize}
\end{theorem}

\begin{remark}
Let $\mathcal{A}$ be an $\omega$-homogeneous relational structure such that $G$ is dense in its automorphism group. Since finitely generated substructures of $\mathcal{A}$ are finite, $(ii)$ of $(b)$ simply sais that $\age(\mathcal{A})$ satisfies the Ramsey property.
\end{remark}

\begin{proof}
$(a)\Rightarrow (c)$ Since $G$ is extremely amenable, it has a fixed point of its natural action on $\lo(\kappa),$ hence $(i')$ holds. To prove $(ii),$ suppose that $c:G/G_A\to\{1,2,\ldots,k\}$ is a colouring of left cosets of $G_A$ by $k$ many colours. Consider $c$ as a point in the compact space $X=\{1,2,\ldots,k\}^G$ and an action of $G$ on $X$ given by $gx(hG_A)=x(g^{-1}hG_A).$ Let $Y$ be the closure of the orbit of $c$ in $X$. Since $G$ is extremely amenable, the induced action of $G$ on $Y$ has a fixed point $d.$  As $G$ acts transitively on $G/G_A,$ $d$ must be a constant function, say with range $\{i\}\subset\{1,2,\ldots,k\}$ Let $B\supset A$ be finite and let $H=\{hG_A\in G: hA\subset B\}\subset G/G_A.$ Since $d\in\overline{Gc},$ there is a $g\in G$ such that $g^{-1}c|H=d|H.$ Then $c(ghG_A)=g^{-1}c(hG_A)=d(hG_A)$ for every $h\in H.$

$(c)\Rightarrow (b)$ Let $<$ be an ordering given by $(i').$ It means that for every $g\in G_{(A)}$ we have that $g(A,<|A)=(A,<|A),$ hence $g(a)=a$ for all $a\in A$ and so $G_{(A)}=G_A.$

$(b)\Rightarrow (a)$ By the proof of Theorem $\ref{T}$ we know, that left translates of a set $H=G_AK$ cannot generate a syndetic subalgebra of $L$ if $H\cap G_{(A)}g=\emptyset$ for some $g\in G.$ However, since $G_{(A)}=G_A,$ it follows that $H$ only generates a syndetic algebra if it is equal to all of $G.$
\end{proof}

As an immediate consequence of the Theorem \ref{T2}, we obtain examples of extremely amenable groups as groups of automorphisms of structures:
\begin{theorem}[\cite{P2}]
The group of automorphisms of an $\omega$-homogeneous linear order is extremely amenable.
\end{theorem}

The following result generalizes Theorem 6.14 of \cite{KPT} to uncountable Boolean algebras.

\begin{theorem} \label{T3}
Let $\mathcal{B}$ be a homogeneous Boolean algebra and let $<$ be a normal ordering induced by the class of naturally ordered finite Boolean algebras as in the Example \ref{natural}. Then $\aut(\mathcal{B},<)$ is extremely amenable.
\end{theorem}

\begin{proof}
By minimality of $\no(B)$ we have that $\age(\mathcal{B},<)$ is the class of all naturally ordered finite Boolean algebras satisfying the Ramsey property. Therefore $(c)$ of the Theorem \ref{T2} is satisfied.
\end{proof}

The following result generalizes Theorem 6.13 of \cite{KPT} to uncountable dimensional vector spaces.

\begin{theorem} \label{T4}
Let $\mathcal{V}$ be a vector space over a finite field and let $<$ be a normal ordering induced by the class of naturally ordered finite vector spaces as in the Example \ref{natural}. Then $\aut(\mathcal{V},<)$ is extremely amenable.
\end{theorem}

\begin{proof}
Identical to the proof of the Theorem \ref{T3}.
\end{proof}

\section*{Acknowledgement}
I am grateful to Bohuslav Balcar who insisted that I study the greatest ambit as the quotient space of the ultrafilter system. I wish to thank  also Stevo Todor\v{c}evi\'c for suggesting to extend some of the results of \cite{KPT} to uncountable structures. 

\bibliographystyle{alpha}
\bibliography{Linear}

\end{document}